\documentclass[12pt]{amsart}
\usepackage{amsmath,amssymb}
\newcommand{\IZ}{\mathbb{Z}}
\newcommand{\IN}{\mathbb{N}}
\newcommand{\IC}{\mathbb{C}}
\newcommand{\IB}{\mathbb{B}}

\newcommand{\cH}{\mathcal{H}}
\newcommand{\cK}{\mathcal{K}}
\newcommand{\cU}{\mathcal{U}}
\newcommand{\ve}{\varepsilon}
\newcommand{\vp}{\varphi}
\newcommand{\Ga}{\Gamma}
\newcommand{\acts}{\curvearrowright}
\DeclareMathOperator{\supp}{\mathop{supp}}
\DeclareMathOperator{\Ball}{Ball}
\DeclareMathOperator{\graph}{graph}
\DeclareMathOperator{\diag}{diag}
\DeclareMathOperator{\dist}{dist}
\DeclareMathOperator{\Prob}{Prob}
\DeclareMathOperator{\id}{id}
\DeclareMathOperator{\Ghost}{\mathrm{G}^*\!}
\DeclareMathOperator{\HS}{\mathrm{HS}}
\DeclareMathOperator{\Alg}{\mathrm{C}^*\!}
\DeclareMathOperator{\Roe}{\mathrm{C}^*_{\mathrm{u}}\!}
\DeclareMathOperator{\red}{\mathrm{C}^*_{\mathrm{red}}\!}
\DeclareMathOperator{\RF}{\mathrm{C}^*_{\mathrm{RF}}\!}
\newcommand{\abs}[1]{\left\lvert#1\right\rvert}
\newcommand{\norm}[1]{\left\lVert#1\right\rVert}
\newcommand{\set}[2]{\{\, #1 \mid #2\,\}}
\newcommand{\setmid}{\mathrel{}\middle|\mathrel{}}
\newcommand{\inpro}[1]{\mathopen{\langle}#1\mathclose{\rangle}}
\newcommand{\tensor}[2]{#1\otimes {#2}^{\mathrm{op}}}
\newtheorem{thmA}{Theorem}

\newtheorem{corA}[thmA]{Corollary}
\newtheorem{thm}{Theorem}
\newtheorem{prop}[thm]{Proposition}
\newtheorem{lem}[thm]{Lemma}
\theoremstyle{definition}
\newtheorem{defn}[thm]{Definition}
\theoremstyle{remark}
\newtheorem{rem}[thm]{Remark}
\newtheorem{ex}[thm]{Example}
\title[C*-exactness and property A]{$\mathrm{C}^*$-exactness and property A for group actions}
\author{Hiroto Nishikawa}
\address{RIMS, Kyoto University, \mbox{606-8502} Japan}
\email{hnishika@kurims.kyoto-u.ac.jp}

\date{\today}

\begin{document}
\begin{abstract}
    For an action of a discrete group $\Ga$ on a set $X$, we show 
    that the Schreier graph on $X$ has property A if and only if the permutation representation on $\ell_2X$ generates an exact $\Alg$-alge\-bra. 
    This is well known in the case of the left regular action on $X=\Ga$ as the equivalence of $\Alg$-exactness and property A of its Cayley graph. 
    This also generalizes Sako's theorem, which states 
    that exactness of the uniform Roe alge\-bra $\Roe(X)$ characterizes property A of $X$ when $X$ is uniformly locally finite.
\end{abstract}
\maketitle
\section{Introduction}
Guoliang Yu introduced \emph{Property A} as the amenable-type condition on metric spaces in~\cite{Yu00}. 
It is characterized in terms of operator alge\-bras by Skandalis, Tu, and Yu in~\cite[Theorem 5.3]{STY02}, 
that is property A of a \emph{uniformly locally finite} (or \emph{ulf} in short) metric space $X$ is equivalent to nuclearity of the \emph{uniform Roe alge\-bra} $\Roe(X)$. 
Moreover, by Sako's remarkable result (\cite[Theorem 1.1]{Sak20}), exactness, as well as local reflexivity of $\Roe(X)$, also characterizes property A under the ulf assumption.
These results also hold when $X$ is a ulf \emph{coarse space}. See~\cite{Sak13a}.

For a finitely generated group $\Ga$, we equip $\Ga$ with the word metric (or generally the coarse structure) and regard it as a ulf metric space $\abs{\Ga}$. 
In the case of discrete groups, it is well-known that property A is equivalent to $\Alg$-exactness, i.e. the reduced group $\Alg$-alge\-bra $\mathrm{C}^*_{\mathrm{red}}(\Ga)$ being exact.
Kirchberg and Wassermann introduced the notion of exact groups in terms of reduced crossed products and characterized this by exactness of $\mathrm{C}^*_{\mathrm{red}}(\Ga)$ (\cite[Theorem 5.2]{KW99}). 
Ozawa proved that this is equivalent to nuclearity of $\Roe\abs{\Ga}$ in~\cite[Theorem 3]{Oza00}, which is identical with property A of $\abs{\Ga}$ (\cite[Theorem 1.1]{HR00},~\cite{ADR00}).

The purpose of this paper is to unify and generalize these results. 
\begin{thmA}[Main Theorem]\label{thm:A}
    For a discrete group $\Ga$ and an action on a set $X$, the Schreier graph $\abs{\Ga\acts X}$ has property A
    if and only if the $\Alg$-alge\-bra $\Alg(\lambda_X(\Ga))$ generated by the permutation representation $\lambda_X:\Ga\to \IB(\ell_2X)$ is exact.

    Moreover, for a set $\Ga$ of partial translations on a set $X$, property A of the ulf coarse structure on $X$ generated by $\Ga$ is equivalent to exactness of the $\Alg$-algebra $\Alg(\Ga)$ generated by $\Ga$ in $\IB(\ell_2X)$.
\end{thmA}
We prove this \emph{partial translation} version in Section~\ref{appendix} by the $2\times2$ trick. We remark that our theorem contains Sako's result about exactness of $\Roe(X)$ because every ulf coarse space is realized as the Schreier graph for some group action.
Although it is not apparent whether local reflexivity of $\Alg(\lambda_X(\Ga))$ characterizes property A or not, 
we give another proof of Sako's result about local reflexivity of $\Roe(X)$ in Corollary~\ref{cor:Sako}. 

It is also studied exactness of the $\Alg$-algebras $\Alg(\mathcal{T})$ generated by a set $\mathcal{T}$ of some partial translations especially when $X$ has uniform embedding to a discrete group and $\mathcal{T}$ comes from the embedding in~\cite{BNW07}. Our main theorem is the definitive result in this direction and all statements in~\cite[Theorem 29]{BNW07} are equivalent without any assumptions on $X$ and $\mathcal{T}$. Here, we only consider a set of partial translations $\mathcal{T}$ which generates $X$ as a coarse structure.

We prove the main theorem in the same spirit as Sako's paper~\cite{Sak20}, which discovers the importance of a weighted trace and treats it like an amenable trace.
It does not seem that CPAP rephrasing works well, in contrast to the group case~\cite{Oza00}.

We explain our method in this toy example.
\begin{ex}[Box space]
    Let $\Ga$ be a residual finite group and $\Lambda_1\supset\Lambda_2\supset\dots$ be finite index normal subgroups of $\Ga$ whose intersection $\bigcap_m \Lambda_m$ is $\{1\}$. 
    The \emph{box space} $\Box\Ga$ is the disjoint union of finite quotients $\bigsqcup_m \Ga/\Lambda_m$ introduced in~\cite[Definition 11.24]{Roe03}. 
    It is observed that $\Box\Ga$ has property A if and only if $\Ga$ is amenable by Guentner, 
    and that exactness of $\Roe(\Box\Ga)$ also characterizes amenability of $\Ga$ by Willett (the last sentence of~\cite{AGS12}). Since this is one of the motivations in~\cite{Sak13b},\cite{Sak20}, we explain this argument in detail.

    First, we consider $\RF(\Ga):=\overline{\IC[\Ga]}\subset\Roe(\Box\Ga)$. By residual finiteness, the canonical trace $\tau$ on $\RF(\Ga)$ is an amenable trace.
    Since $\RF(\Ga)$ is exact, $\tau$ on $\red(\Ga)$ is also amenable, which is equivalent to amenability of $\Ga$ by using the characterization of an amenable trace in~\cite[Theorem 6.2.7]{BO08}. Along this line, Sako proved that exactness of $\Roe(X)$ implies property A of $X$ in~\cite{Sak20}.

    In this case, our method is to use directly the min-continuity of \[\tau:\tensor{\red(\Ga)}{\RF(\Ga)}\to\IC,\ \ \tau(\tensor{a}{b}):=\tau(ab),\]
    which comes from exactness of $\RF(\Ga)$. By Theorem~\ref{HahnBanach}, one has vectors $\zeta^i\in \ell_2\Ga\otimes\overline{\ell_2X}\otimes \ell_2$ that approximate $\tau$ as the vector states. 
    Taking the norm, we reduce $\zeta^i$ to $\eta^i\in \ell_2\Ga$, which are approximately invariant vectors of $\Ga$ by Lemma~\ref{lem}. So $\Ga$ is amenable.
\end{ex}
\subsection*{Acknowledgements}
The author would like to thank Narutaka Ozawa for his support and helpful comments as my supervisor. The author also thanks anonymous reviewers for many useful suggestions and presenting the reference~\cite{BNW07}.
\section{Preliminaries}
Although we treat general coarse spaces, we reduce everything to the graph case 
by regarding general coarse structures as the directed increasing union of graphs. 
We give a brief introduction to ulf coarse spaces in Section~\ref{appendix}. For more details, see standard references on coarse geometry (\cite{Roe03},~\cite{NY12}).
    \subsection{Definitions}
    For an action of a finitely generated group $\Ga$ on a set $X$, define the \emph{Schreier graph} $\abs{\Ga\acts X}$ as the graph $(X,E)$, where $E:=\set{(sx,x)\in X\times X}{x\in X, s\in S}$ and $S\subset\Ga$ is a symmetric finite generator set. 
    This coincides with the Cayley graph when the action is the left regular action $\Ga\acts\Ga$. As with the Cayley graph, the coarse structure of the Schreier graph does not depend on the choice of $S$.
    
    Let $(X,E)$ be a (possiblly non-connected) undirected graph and $d$ be its graph metric ($d$ takes a value in $[0,\infty]$). 
    We assume that $\sup_{x\in X}\abs{E\cap(X\times\{x\})} <\infty$, which is called \emph{uniformly locally finite}, or \emph{ulf} in short
    (also known as of uniformly bounded degree, of bounded geometry). For $R>0$, we denote by
    \[\IC_\mathrm{u}^{\leq R}(X):=\left\{\, a\in\IB(\ell_2X) \setmid a(x,y)=\inpro{\delta_x,a\delta_y}=0\ \text{if}\ d(x,y)>R\,\right\}\]
    the linear space of $R$-propagation operators
    and define the \textit{uniform Roe alge\-bra} $\Roe(X)$ as norm closure of $\bigcup_{R>0}\IC_\mathrm{u}^{\leq R}(X)$. 
    
    For a Hilbert space $\cH$, denote by $\IB(\cH), \HS(\cH)$  the $\Alg$-alge\-bra of all the bounded operators and the Hilbert space of all the Hilbert--Schmidt operators, respectively. 
    Let $\Prob(X)\subset \ell_1X$ be the space of probability measures on $X$. For $T\in\IB(\ell_2X)$, we denote by $T(x,y)$ $(x,y)$-matrix entries of $T$. We also use the same notation for $\HS(\ell_2X)$ and $\ell_2X$. Simply write the spatial tensor product on $\Alg$-algebras as $\otimes$.
    \subsection{Property A}
    Property A has many characterizations, like amenability in the group theory. We refer to some of them. See~\cite{Yu00} for the original definition.
    \begin{defn}[Guoliang Yu]
        $T\in\Roe(X)$ is called a \emph{ghost} if $\Big(T(x,y)\Big)_{x,y}\in c_0(X\times X)$.
    \end{defn}
    Every compact operator is ghost and ghosts form an ideal $\Ghost(X)$ of $\Roe(X)$~\cite{RW14}.
    \begin{thm}[See~{\cite[Theorem 5.3]{STY02}},~{\cite[Theorem 1.3]{RW14}}]
        For a ulf metric space $X$, the following are all equivalent.
        \begin{itemize}
            \item $X$ has property A.
            \item Every ghost on $X$ is compact.
            \item $\Roe(X)$ is nuclear.
        \end{itemize}
    \end{thm}
    This theorem also holds when $X$ is a ulf coarse space. See~\cite{Sak13a} and~\cite{RW14}.
    \begin{prop}[See{~\cite[Proof of Lemma 4.2]{RW14}} for being sparse]
        For a ulf graph $(X, E)$ without property A, there exists a non-compact ghost $T\in\Roe(X)$. Moreover, we may assume $T$ to be a \emph{sparse diagonal} matrix:
        \begin{itemize}
            \item There exist sparse disjoint subsets $V_n\Subset X$ such that $T=\mathrm{SOT}$-$\sum_n P_{V_n}TP_{V_n}$,
            \item $T_n:=P_{V_n}TP_{V_n}$ is positive and of norm 1,
        \end{itemize}
        where we call $\{V_n\}_n$ \emph{sparse} if $\dist(V_n,V_m)\to\infty$ as $n,m\to\infty$ and $n\neq m$.
    \end{prop}
    \begin{lem}
        Let $(X, E)$ be a ulf graph and $\ve>0$. Then there exists $L>1$ 
        such that for every $\eta\in\Prob(X)$ with a finite support, there exists $\eta'\in \Prob(X)$ satisfying that
        \begin{itemize}
            \item $\norm{\eta-\eta'}_1\leq\ve$,
            \item $L^{-1}\eta'(x)\leq\eta'(y)\leq L\eta'(x)$ for every $(x,y)\in E$,
            \item $\norm{P_V\eta'}\leq\ve^{\dist(V,\supp\eta)}$ for $V\subset X$, where $P_V$ stands for the canonical projection from $\ell_1X$ onto $\ell_1V$.
        \end{itemize}
    \end{lem}
    This lemma is essentially proven in~\cite[Theorem 3.17 (4),(5)]{Sak20}.
    \begin{proof}
        Define $M:=\sup_{x\in X}\abs{E\cap(\{x\}\times X)}<\infty$ and $L:=\max\{1,M(1+\ve^{-1})\}$.
        \[\eta'(x):=\sum_w L^{-d(x,w)}\eta(w)\]
        satisfies the three conditions above because
        \begin{itemize}
            \item $\eta'\geq\eta$ and $\sum_x L^{-d(x,w)}\leq1+\epsilon$ for every $w\in X$,
            \item $d(x,w)+1\geq d(y,w)\geq d(x,w)-1$ for every $(x,y)\in E$,
            \item $\sum_{d(x,w)\geq R}L^{-d(x,w)}\leq \epsilon^R(1+\epsilon)^{-R+1}$ for every $w\in X, R>0$.
        \end{itemize}
        Then we adjust $\eta'$ to have norm 1 and $L$ to be a bit larger.
    \end{proof}
    For the sparse diagnoal ghost $T=(T_n)_n$, take $\xi_n\in \ell_2V_n$ such that $T_n\xi_n=\xi_n$. By conjugating $T$ by a unitary in $\ell_\infty X$, we may assume $\xi_n$ to have non-negative entries. 
    Using the previous lemma for $(\xi_n)^2$, we take $\xi'_n\in \ell_2X$ so that $\norm{(1-P_{V'_n})\xi'_n}\to0$ for a bit larger sparse disjoint subsets $\{V'_n\}$ due to the exponential decay of $(\xi'_n)^2$. 

    Assume that $(X,E)$ is realized as a Schreier graph of some $\Ga\acts X$. For each $\gamma\in\Ga$, define $h^\gamma\in \ell_\infty X$ by $h^\gamma(x)\xi_n(\gamma^{-1}x)=\xi_n(x)$ for $x\in \bigsqcup_nV'_n$ and $h^\gamma(x):=1$ otherwise, 
    then $h^\gamma\in \ell_\infty X$ is invertible and $\norm{h^\gamma\gamma\xi_n-\xi_n}\to0$ as $n\to\infty$. So we obtain the following:
    \begin{prop}\label{notation}
        Let $\Ga\acts X$ be a discrete action whose Schreier graph on $X$ does not have property A and $\ve>0$. 
        Then there exists a positive ghost $T\in\Roe(X)$ of norm 1, a sequence of non-negative unit vectors $\xi_n\in \ell_2X$ which converges to 0 in weak topology, and invertible elements $\{h^\gamma\}_\gamma\subset\ell_\infty X$ satisfying that
        \begin{itemize}
            \item $\norm{T\xi_n-\xi_n}\leq\ve$,
            \item $\norm{h^\gamma\gamma\xi_n-\xi_n}\to0$ as $n\to\infty$.
        \end{itemize}
    \end{prop}
\section{Proof of Main Theorem}
To prove the main theorem, we mimic the proof in~\cite{Sak20} for $\bigsqcup_n V'_n$ as a \emph{generalized box space} in~\cite{Sak13b}. 
The key representation $(\pi,\cH,\Xi)$ appeared in~\cite{Sak20} as $(\pi_\infty,\cH_\infty,\Xi_\infty)$ is an analogy of the left regular representation in the box space case. 
The compression $\Phi_S$ in~\cite{Sak14} corresponds to the trace in that case and almost factors through $\pi$. We use the Hahn--Banach theorem to the vector state of $\Xi$.
    \subsection{Notations}
    For vectors $\xi,\xi'$ (or just complex numbers), we write $\xi\approx_\ve\xi'$ when $\norm{\xi-\xi'}\leq\ve$. 
    Denote by $\vp_\xi$ the vector state on $A$ by a unit vector $\xi\in\cH$ when a $\Alg$-alge\-bra $A\subset\IB(\cH)$ is concretely represented.
    
    For $S>0,x\in X$, we recall the compression $\Phi_S$ in~\cite{Sak14}
    \[\Phi_S:\Roe(X)\to \prod_{x\in X}\IB(\ell_2\Ball(x,S)),\ \Phi_S(a):=\left( P_{\Ball(x,S)}aP^*_{\Ball(x,S)}\right)_x,\]
    where $P_V$ stands for the orthogonal projection from $\ell_2X$ onto $\ell_2V$ for $V\subset X$ and $\Ball(x,S):=\set{y\in X}{d(x,y)\leq S}$. 
    
    Let $\cU$ be a (non-principle) ultrafilter on $\IN$, $\prod_\cU\HS(\ell_2X)$ be the ultrapower and $\mathrm{weak\text{-}}\lim_\cU:\prod_\cU\HS(\ell_2X)\to \HS(\ell_2X)$ be the map taking the weak limit via $\cU$. Define
    \[\cH_0:=\left(\bigcup_{R>0}\prod_\cU\HS^{\leq R}(\ell_2X)\right)\cap\ker\left(\mathrm{weak\text{-}}\lim_\cU\right)\ \subset\ \prod_\cU\HS(\ell_2X),\]
    where $\HS^{\leq R}(\ell_2X):=\HS(\ell_2X)\cap\IC^{\leq R}_{\mathrm{u}}(X)$ is the Hilbert space of $R$-propagation Hilbert--Schmidt operators and $\cH$ as the norm closure of $\cH_0$.
    Consider the left-right representation $\pi'$ of $\tensor{\Roe(X)}{\Alg(\lambda_X(\Ga))}$ on $\HS(\ell_2X)=\ell_2X\otimes\overline{\ell_2X}$, that is $\pi'(\tensor{a}{b})\Omega:=a\Omega b$ for $\Omega\in\HS(\ell_2X)$. 
    Since the ultrapower of $\pi'$ preserves finiteness of the propagation and $\mathrm{weak-}\lim_\cU$, we obtain the diagonal representation $\pi:\tensor{\Roe(X)}{\Alg(\lambda_X(\Ga))}\acts\cH$.
    
    Since we assume the negation of property A of $\abs{\Ga\acts X}$, we fix $T\in\Roe(X),\xi_n\in\ell_2X$ and $h^\gamma\in\ell_\infty X$ as in Lemma \ref{notation} throughout this section. 
    The vectors $\Xi_n:=\diag(\xi_n)\in\HS(\ell_2X)$ and $\Xi:=(\Xi_n)_{\cU}\in \cH\subset\prod_\cU\HS(\ell_2X)$ are unit vectors such that $(h^\gamma\gamma\xi_n)_{\cU}=(\xi_n)_{\cU}$ in the ultrapower.
    
    The vector state $\tau:=\vp_\Xi$ on $\tensor{\Roe(X)}{\Alg(\lambda_X(\Ga))}$ plays an important role as in the theory of amenable traces.
    $A:=\tensor{\pi_{\mathrm{left}}(\Roe(X))}{\Alg(\lambda_X(\Ga))}$ is faithfully represented on $\cH\otimes\overline{\ell_2X}$, 
    where we write $\pi_{\mathrm{left}}$ by the restriction of $\pi$ to $\Roe(X)\otimes\IC1$. If $\Alg(\lambda_X(\Ga))$ is exact, the following short sequence
    \[0\to\tensor{\ker\pi_{\mathrm{left}}}{\Alg(\lambda_X(\Ga))}\to\tensor{\Roe(X)}{\Alg(\lambda_X(\Ga))}\to\tensor{\pi_{\mathrm{left}}(\Roe(X))}{\Alg(\lambda_X(\Ga))}\to0\]
    is exact and $\pi$ factors through $A=\tensor{\pi_{\mathrm{left}}(\Roe(X))}{\Alg(\lambda_X(\Ga))}$. So $\tau$ is the state on $A$, which concretely represented on $\cH\otimes\overline{\ell_2X}$. As with the box space case, we only use exactness of $\Alg(\lambda_X(\Ga))$ in this step.
\subsection{Proof}
 \setcounter{thmA}{0}
\begin{thmA}[Main Theorem]
    For a discrete group $\Ga$ and an action on a set $X$, the Schreier graph $\abs{\Ga\acts X}$ has property A
    if and only if the $\Alg$-alge\-bra $\Alg(\lambda_X(\Ga))$ generated by the permutation representation $\lambda_X:\Ga\to \IB(\ell_2X)$ is exact.
\end{thmA}
First, property A of $\abs{\Ga\acts X}$ implies exactness of $\Alg(\lambda_X(\Ga))$ since it is a subalge\-bra of the nuclear $\Alg$-alge\-bra $\Roe(X)$.

We only prove this theorem when $\Ga$ is finitely generated. Then the general case follows immediately 
because exactness passes to a subalge\-bra $\Alg(\lambda_X(\Lambda))\subset \Alg(\lambda_X(\Ga))$ and nuclearity is preserved under increasing union $\Alg(\ell_\infty X,\lambda_X(\Ga))=\lim_{\Lambda\in\mathcal{F}}\Alg(\ell_\infty X,\lambda_X(\Lambda))$, 
where $\mathcal{F}$ is the directed family of finitely generated subgroups of $\Ga$. 
Note that nuclearity of $\Alg(\ell_\infty X,\lambda_X(\Lambda))=\Roe\abs{\Lambda\acts X}$ is equivalent to property A of the Schreier graph $\abs{\Lambda\acts X}$ w.r.t. the restricted action of $\Lambda$ on $X$.

We assume that $\Alg(\lambda_X(\Ga))$ is exact and $\abs{\Ga\acts X}$ does not have property A and $\Ga$ is finitely generated.
Thus, we treat the Schreier graph $\abs{\Ga\acts X}$ as a graph.

We use the Hahn--Banach theorem (or a weak form of the Glimm's lemma) for $A:=\tensor{\pi_{\mathrm{left}}(\Roe(X))}{\Alg(\lambda_X(\Ga))}\subset\IB(\cH\otimes\overline{\ell_2X})$ and $\tau$ on $A$.
\begin{lem}\label{HahnBanach}
    Let $A\subset\IB(\cH)$ be a $\Alg$-subalge\-bra. Then the set of vector states $\vp_\zeta$ for $\zeta\in\cH\otimes \ell_2$ is dense in the state space of $A$ w.r.t. weak${}^*$-topology of $A^*$.
\end{lem}
Take unit vectors $\zeta^i\in\cH\otimes\overline{\ell_2X}\otimes \ell_2$ such that $\vp_{\zeta^i}\to \tau=\vp_\Xi$. Define $\cK:=\overline{\ell_2X}\otimes \ell_2$. Since we may assume that $\zeta^i$ are in the alge\-braic tensor product $\cH_0\odot\cK$, take lifts $(\zeta^i_n)_n$ of $\zeta^i$ in $\prod_n\Big(\HS(\ell_2X)\odot\cK\Big)$, and define unit vectors $\eta^i\in\cH_0$ as $\eta^i_n(x,y):=\norm{\zeta^i_n(x,y)}_\cK, \eta^i:=(\eta^i_n)_{\cU}$, where $\zeta^i_n(x,y)\in \cK$ is the $(x,y)$-entry of $\zeta^i_n\in\HS(\ell_2X)\otimes\cK$.

Let us now sort out what vectors define the vector states on which alge\-bras. The vectors $\Xi\in\cH$ and $\zeta^i\in\cH\otimes\cK$ define states on $A$, $\xi_n\in \ell_2X$ and $\eta^i\in\cH$ define states on $\Roe(X)$.
\begin{lem}\label{lem}
    For $\ve>0,\ \gamma\in\Ga,\ f\in \ell_\infty X\subset\Roe(X)$, and sufficiently large $i$, the following hold.
    \begin{itemize}
        \item $\inpro{\eta^i,h^\gamma\gamma\eta^i}\geq\abs{\vp_{\zeta^i}(\tensor{h^\gamma\gamma}{\gamma^{-1}})}\approx_\ve \vp_\Xi(\tensor{h^\gamma\gamma}{\gamma^{-1}})=1$.
        \item $\norm{h^\gamma\gamma\eta^i}^2=\vp_{\zeta^i}(\abs{\tensor{h^\gamma\gamma}{\gamma^{-1}}}^2)\approx_\ve \vp_\Xi(\abs{\tensor{h^\gamma\gamma}{\gamma^{-1}}}^2)=1$.
        \item $\inpro{\eta^i,f\eta^i}=\vp_{\zeta^i}(f)\approx_\ve\vp_\Xi(f)=\lim_{\cU}\vp_{\xi_n}(f)$.
    \end{itemize}
\end{lem}
\begin{proof}
    Note that $(\tensor{h^\gamma\gamma}{\gamma^{-1}})\Xi=h^\gamma\gamma\Xi\gamma^{-1}=\Xi$ and $\inpro{\Xi_n,f\Xi_n}=\inpro{\xi_n,f\xi_n}$. One has
    \begin{align*}
        \inpro{\eta^i,h^\gamma\gamma\eta^i} &= \lim_{\cU}\sum_{x,y}h^\gamma(x)\eta^i_n(x,y)\eta^i_n(\gamma^{-1}x,y)\\
        &\geq\lim_{\cU}\abs{\sum_{x,y}h^\gamma(x)\inpro{\zeta^i_n(x,y),\overline{\gamma}\zeta^i_n(\gamma^{-1}x,y)}_{\cK}}\\
        &=\abs{\inpro{\zeta^i,(\tensor{h^\gamma\gamma}{\gamma^{-1}})\zeta^i}},
    \end{align*}
    \begin{align*}
        \norm{h^\gamma\gamma\eta^i}^2 &= \lim_{\cU}\sum_{x,y}h^\gamma(x)^2\eta^i_n(\gamma^{-1}x,y)^2\\
        &=\lim_{\cU}\sum_{x,y}h^\gamma(x)^2\norm{\overline{\gamma}\zeta^i_n(\gamma^{-1}x,y)}^2_{\cK}\\
        &=\norm{(\tensor{h_\gamma\gamma}{\gamma^{-1}})\zeta^i}^2,
    \end{align*}
    \begin{align*}
        \inpro{\eta^i,f\eta^i} &= \lim_{\cU}\sum_{x,y}f(x)\eta^i_n(x,y)^2\\
        &=\lim_{\cU}\sum_{x,y}f(x)\norm{\zeta^i_n(x,y)}^2_{\cK}\\
        &=\inpro{\zeta^i,f\zeta^i},
    \end{align*}
    and $\vp_{\zeta^i}(x)\approx_\ve \tau(x)=\vp_{\Xi}(x)$ for $x\in A$ and suffciently large $i$.
\end{proof}
Thus, $h^\gamma\gamma\eta^i\approx_{2\sqrt{\epsilon}}\eta^i$ and $\vp_{\eta^i}(f)\approx_{\ve}\lim_{\cU}\vp_{\xi_n}(f)$ for $f\in \ell_\infty X$. One has
\begin{align*}
    \vp_{\eta^i}(fh^\gamma\gamma) &= \inpro{\eta^i,fh^\gamma\gamma\eta^i}\\
    &\approx_{2\sqrt{\epsilon}}\inpro{\eta^i,f\eta^i}\\
    &\approx_{\ve}\lim_{\cU}\inpro{\xi_n,f\xi_n}\\
    &=\lim_{\cU}\inpro{\xi_n,fh^\gamma\gamma\xi_n}\\
    &=\lim_{\cU}\vp_{\xi_n}(fh^\gamma\gamma).
\end{align*}
Therefore, $\vp_{\eta^i}\to\lim_{\cU}\vp_{\xi_n}$ in the weak${}^*$-topology of $\Roe(X)^*$.
\begin{lem}
    For every unit vector $\eta\in\cH_0$, there exists $S>0$ such that
    $\abs{\vp_\eta(a)}\leq\norm{\Phi_S(a)\mod\bigoplus_{x\in X}\IB(\ell_2\Ball(x,S))}$ for every $a\in\Roe(X)$.
\end{lem}
This lemma is observed in \cite[Theorem 3.17]{Sak20}.
\begin{proof}
    Take $S>0$ as $\eta\in\prod_{\cU}\HS^{\leq S}(\ell_2X)$.
    Define $\omega^y_n\in \ell_2\Ball(y,S)$ as $\omega^y_n(x):=\eta_n(x,y)$ and unit vectors $\omega_n:=\oplus_y\omega^y_n\in\bigoplus_y \ell_2\Ball(y,S)$. 
    \[\vp_\eta(a)=\lim_{\cU}\inpro{\eta_n,a\eta_n}_{\HS(X)}=\lim_{\cU}\sum_y\inpro{\omega^y_n,a\omega^y_n}_{\ell_2X}=\lim_{\cU}\inpro{\omega_n,\Phi_S(a)\omega_n}_{\bigoplus\ell_2\Ball(y,S)}\]
    shows that $\vp_\eta$ factors through $\Phi_S$. Since $\text{weak-}\lim_\cU\eta_n=0$, so does $\omega_n\rightharpoonup0\in\bigoplus_y \ell_2\Ball(y,S)$. Note that compact operators vanish on $\lim_\cU \vp_{\omega_n}$.
\end{proof}
\begin{proof}[Proof of Main Theorem.]
    For the ghost $T\in\Roe(X)$, one has $\Phi_S(T)\in\bigoplus_{x\in X}\IB(\ell_2\Ball(x,S))$ for all $S>0$ and $\vp_{\eta^i}(T)=0$ for all $i$ by the previous lemma. Recall the ghost $T$ satisfies $\lim_{\cU}\vp_{\xi_n}(T)\approx_\ve1$ and this gives a contradiction to the weak${}^*$-convergence $\vp_{\eta^i}\to\lim_\cU\vp_{\xi_n}$. We obtain the main theorem.
\end{proof}
\begin{rem}
    Note that the kernel of the left representation $\pi_{\mathrm{left}}:\Roe(X)\acts\cH$ is $\Ghost(X)$, the ideal of ghosts in $\Roe(X)$ by the previous lemma.
\end{rem}
\begin{proof}
    As the previous lemma, we prove the following equality for $S>0, a\in(\Roe(X))_+$:
    \begin{align*}
        &\sup\Big\{\vp_\eta(a):\eta\in\cH\cap\prod_{\cU}\HS^{\leq S}(\ell_2X),\norm{\eta}=1\Big\}\\
        =&\sup\Big\{\lim_\cU\inpro{\omega_n,\Phi_S(a)\omega_n}:\omega_n\in\bigoplus_y \ell_2\Ball(y,S),\norm{\omega_n}=1,\omega_n\rightharpoonup0 \Big\}\\
        =&\norm{\Phi_S(a)\mod\bigoplus_{x\in X}\IB(\ell_2\Ball(x,S))}.
    \end{align*}
    Indeed, for $\eta\in\prod_{\cU}\HS^{\leq S}(\ell_2X)$, $\eta\in\cH$ is equivalent to $\mathrm{weak\text{-}}\lim_\cU\eta=0\in\HS(X)$ and to the weak convergense $\omega_n\rightharpoonup0\in\bigoplus_y \ell_2\Ball(y,S)$.
    So $a\in\ker(\pi_{\mathrm{left}})$ if and only if $\Phi_S(a)\in\bigoplus_{x\in X}\IB(\ell_2\Ball(x,S))$ for all $S>0$, which means $\Big(a(x,y)\Big)_{x,y}\in c_0(X\times X)$. Note that $a$ is an approximately finite propagation operator.
\end{proof}
\begin{rem}
     In the proofs, we only use exactness for 
    \[\tensor{(\Roe(X)/\Ghost(X))}{\Alg(\lambda_X(\Ga))}=\left(\tensor{\Roe(X)}{\Alg(\lambda_X(\Ga))}\right)\Big/\left(\tensor{\Ghost(X)}{\Alg(\lambda_X(\Ga))}\right).\]
    Since we can use $\mathrm{C}_{\mathrm{full}}^*(\Ga)$ instead of $\Alg(\lambda_X(\Ga))$ in the min-continuity, the continuity of
    \[\tensor{(\Roe(X)/\Ghost(X))}{\mathrm{C}_{\mathrm{full}}^*(\Ga)}\to\left(\tensor{\Roe(X)}{\Alg(\lambda_X(\Ga))}\right)\Big/\left(\tensor{\Ghost(X)}{\Alg(\lambda_X(\Ga))}\right)\]
    implies property A of $\abs{\Ga\acts X}$.
\end{rem}
\begin{corA}[Sako]\label{cor:Sako}
    For a ulf coarse space $X$, local reflexivity of $\Roe(X)$ implies property A of $X$.
\end{corA}
\begin{proof}
    Realize $X$ as a Schreier graph of some action $\Ga\acts X$. By the previous remark, it suffices to show that local reflexivity of $\Roe(X)$ implies
    \[\tensor{(\Roe(X)/\Ghost(X))}{\Alg(\lambda_X(\Ga))}=\left(\tensor{\Roe(X)}{\Alg(\lambda_X(\Ga))}\right)\Big/\left(\tensor{\Ghost(X)}{\Alg(\lambda_X(\Ga))}\right).\]
    This follows from general facts on $\Alg$-algebras~\cite[Theorem 3.2]{EH85}:\\
    If the $\Alg$-algebra $B$ is locally reflexive and $K$ is an ideal of $B$, then for every $\Alg$-algebra $C$, the naturally defined sequence
    \[0\to K\otimes C\to B\otimes C\to (B/K)\otimes C\to 0\]
    is exact.
\end{proof}
\section{Appendix}\label{appendix}
This appendix shows that every ulf coarse space comes from some group action. 
\begin{defn}[coarse space]
    A family of subsets $\mathcal{E}\subset 2^{X\times X}$ is called a \emph{coarse structure} on a set $X$ when $\mathcal{E}$ forms an ideal by the inclusion order (i.e., being a lower and directed set) and satisfies the following group-like axioms:
    \begin{itemize}
        \item $\Delta_X:=\{(x,x)\in X\times X\ \mid\ x\in X\}\in\mathcal{E}$;
        \item $E\circ F:=\{(x,z)\in X\times X\ \mid\ (x,y)\in E, (y,z)\in F\}\in\mathcal{E}$ for every $E,F\in\mathcal{E}$;
        \item $E^{-1}:=\{(y,x)\in X\times X\ \mid\ (x,y)\in E\}\in\mathcal{E}$ for every $E\in\mathcal{E}$.
    \end{itemize}
\end{defn}
\begin{ex}[extended metric space, graph]
    Let $(X,d)$ be a metric space (possibly $d$ takes a value in $[0,\infty]$). We associate the metric coarse structure
    \[\mathcal{E}:=\{E\subset X\times X\ \mid\ d\ \text{is bounded on}\ E\}.\]
    Especially, when $d$ comes from the graph metric of a graph $(X, E)$, this coincides with the coarse structure generated by $E$, i.e. the smallest coarse structure containing $E$.
\end{ex}
Note that a coarse space $(X,\mathcal{E})$ comes from a graph if and only if $\mathcal{E}$ is finitely generated (actually, singly generated), 
and $(X,\mathcal{E})$ comes from an extended metric if and only if $\mathcal{E}$ is countably generated. 
To see this, take a generating set $\{E_n\}_{n=0}^\infty$ as $E_0=\Delta_X,\ E_n\circ E_m\subset E_{n+m},\ E^{-1}_n=E_n$, and define $d(x,y):=\min\set{n}{(x,y)\in E_n}$.\\
We define the Schreier graph (coarse space) for a general discrete group and its action.
\begin{defn}[Schreier graph]
    For a discrete group action $\Ga\acts X$, we define the Schreier graph $\abs{\Ga\acts X}$ 
    as the ulf coarse structure generated by $\{\graph\gamma\}_{\gamma\in\Ga}$, 
    where $\graph\gamma:=\set{(\gamma x,x)\in X\times X}{x\in X}$
\end{defn}
\begin{defn}[ulf]
    For a coarse space $(X,\mathcal{E})$, we assume that $\sup_x\abs{S\cap(X\times\{x\})}<\infty$ for every $S\in\mathcal{E}$, which is called \emph{uniformly locally finite}, or \emph{ulf} in short.
\end{defn}
\begin{defn}[uniform Roe algebra] The $\Alg$-algebra
    \[\Roe(X,\mathcal{E}):=\overline{\bigcup_{S\in\mathcal{E}}\set{a\in\IB(\ell_2X)}{a(x,y)=0\ \text{if}\ (x,y)\notin S}}\]
    coincides with our definition of the uniform Roe algebra in the graph case.
\end{defn}
\begin{lem}[edge coloring]
    Let $S\subset X\times X$ be a symmetric subset containing $\Delta_X$. If $d:=\sup_{x\in X}\abs{S\cap(\{x\}\times X)}<\infty$, there are involutions $\{\gamma_i\}_{i=1}^{2d-1}$ on $X$ such that $S$ is the union of $\graph \gamma_i$.
\end{lem}
This is well-known. We include a proof for the reader's convenience.
\begin{proof}
    Consider a partially defined symmetric map $c: S\to\{1,\dots,2d-1\}$ as edge $(2d-1)$-coloring (i.e. adjacent edges have distinct colors) and take a maximal one w.r.t. the inclusion order of its domains. We claim that $c$ is totally defined. Suppose not and take $(x,y)$ from complement of its domain. Since $S':=(S\setminus\{(x,y)\})\cap(\{x\}\times X\cup X\times\{y\})$ has the cardinality less than $2d-1$, we take $i\in\{1,\dots,2d-1\}\setminus c(S')$. By extending $c$ as $c(x,y),c(y,x):=i$, we have a contradiction.\\
    Then the reflection $\gamma_i$ on $c^{-1}(i)$ (identity on its complement) works well.
\end{proof}
So the free product of $\IZ/2\IZ$ action gives the following.
\begin{rem}
    Every ulf coarse space is realized as the Schreier graph of (possibly infinitely generated) some group action. 
    When the coarse space comes from a graph (i.e., finitely generated as a coarse space), the group can be taken to be finitely generated.
\end{rem}
\begin{prop}
    Let $\Ga$ be a family of partial translations on a set $X$. Then, $\Alg(\Ga)\subset\IB(\ell_2X)$ is exact if and only if the coarse structure $(X,\mathcal{E})$ generated by $\Ga$ has property A.
\end{prop}
\begin{proof}
    \[U(\gamma):=\begin{pmatrix}
        1-\gamma\gamma^* & \gamma\\
        \gamma^* & 1-\gamma^*\gamma
    \end{pmatrix}\]
    has $\{0,1\}$-valued matrix entries and gives an involution on $X\sqcup X$ for $\gamma\in\Ga$. Let $\Ga_2$ be the group generated by $\{U(\gamma)\}_\gamma$ and $U(\id_X)$. By our main theorem for $\Ga_2\acts X\sqcup X$, exactness of $\Alg(\lambda_{X\sqcup X}(\Ga_2))$ implies property A of $\abs{\Ga_2\acts(X\sqcup X)}$. Since $(X,\mathcal{E})$ is coarsely equivalent to $\abs{\Ga_2\acts(X\sqcup X)}$ and $\Alg(\lambda_{X\sqcup X}(\Ga_2))\subset \mathbb{M}_2\otimes\widetilde{\Alg(\Ga)}$, where $\widetilde{A}$ stands for the unitization of $A$, exactness of $\Alg(\Ga)$ implies property A of $(X,\mathcal{E})$. The converse is straightforward.
\end{proof}

\end{document}